\documentclass[11pt, reqno]{amsart}
  \usepackage{amsmath, amsfonts, amssymb,amsthm}
\numberwithin{equation}{section}
\newtheorem{theorem}{Theorem}
\newtheorem{lemma}{Lemma}
\newtheorem{cor}{Corollary}

\theoremstyle{definition}

\def\min{\mathop{\mathrm{min}}}

\newcommand{\C}{\mathbb{C}}

\begin{document}
\title{Non-Archimedean second main theorem sharing small functions}
\dedicatory{Dedicated to Professor Ha Huy Khoai on the occasion of
  his 75th birthday}     
\author{Ta Thi Hoai An}
\address{Institute of Mathematics, Vietnam Academy of Science and Technology\\
18 Hoang Quoc Viet Road, Cau Giay District \\
10307 Hanoi,  Vietnam}
\address{and: Institute of Mathematics and Applied Sciences (TIMAS)\\ Thang Long University\\ Hanoi,  Vietnam}
\email{tthan@math.ac.vn}
\author{Nguyen Viet Phuong}
\address{Thai Nguyen University of Economics and Business Administration, Vietnam} \email{nvphuongt@gmail.com}

\thanks{Key words: Non-Archimedean field, meromorphic functions, Nevanlinna theory, small functions, uniqueness.}

\thanks{2010 Mathematics Subject Classification 30D35}

\begin{abstract} In this paper, we establish a new second main theorem for meromorphic functions on a non-Archimedean field and small functions with counting functions truncated to level $1.$ As an application, we show that two meromorphic functions on a non-Archimedean field must coincide to each other if they share $q\, (q\geq 5)$ distinct small functions ignoring multiplicities. In particular, if two non-Archimedean meromorphic functions share $5$ small functions ignoring multiplicities, they must be identical. Thus, our work improves the results in \cite{EY}.
\end{abstract}

\thanks{
}

\baselineskip=16truept 
\maketitle 
\pagestyle{myheadings}
\markboth{}{}

\section{ Introduction and main results }
 Let $\mathbf{K}$ be an algebraically closed field of arbitrary characteristic, complete with respect to a non-Archimedean absolute value $|.|.$ For any nonconstant non-Archimedean meromorphic function $f,$ we denote by $S(r,f)$ any quantity satisfying $S(r,f)=o(T(r,f))$ for $r\rightarrow\infty.$ A non-Archimedean meromorphic function $a$ is called a small function with respect to $f$ if $T(r,a)=S(r,f).$ Let $k$ is a positive integer or $\infty,$ we denote by $\overline{E}(a,k,f)$ the set of distinct zeros of $f-a$ with multiplicities at most $k,$ where a zero of $f-\infty$ means a pole of $f.$ In the case that $k=\infty,$ we denote $\overline{E}(a,\infty,f)$ simply by $\overline{E}(a,f).$ We say that $f$ and $g$ {\it share a function $a$ ignoring multiplicities} if $\overline{E}(a,f)=\overline{E}(a,g).$

In 1926, as an application of the celebrated Nevanlinna's value distribution theory of meromorphic functions, R. Nevanlinna \cite{N} himself proved that for two distinct nonconstant meromorphic functions $f$ and $g$ on the complex plane $\C,$ they cannot have the same inverse images for five distinct values. This result is known as Nevalinna’s five values theorem and is considered as the first result on the uniqueness problem of meromorphic functions sharing values. Then, some authors have generalized the result of Nevanlinna to the case where the five distinct values are replaced by five small functions, for example Yuhua and Jianyong \cite{YJ}, Yao \cite{Y}, Yi \cite{Yi}, Thai and Tan \cite{TT}. The proofs of the above results are based straightforwardly on the Cartan’s auxialiary functions and the sharp second main theorem for meromorphic functions and small functions on $\C$ with counting functions truncated to level $1$ of Yamanoi \cite{Ya}. 

Nevanlinna theory in complex analysis is so beautiful that one would naturally be interested in determining how such a theory would look in the $p$-adic analysis. In 1997, Hu and Yang \cite{HY} proved that for two distinct nonconstant $p$-adic meromorphic functions $f$ and $g$ on $\mathbb{C}_p$
  they cannot have the same inverse images for four distinct values. This result can be regarded as the $p$-adic analogue of corresponding uniqueness theorem in value distribution theory of meromorphic functions of one complex variable. More general, this result can be proved for any algebraically closed field $\mathbf{K}$ complete with respect to a non-Archimedean absolute value (cf. \cite{HY1}). Corresponding to such a uniqueness theorem of non-Archimedean meromorphic functions which have the same inverse images for distinct values, one expects a generalized result in which in this case the distinct values replaced by the small functions. However, as we known, there is no second main theorem similar to the theorem of Yamanoi in non-Archimedean analysis. Recently, A. Escassut and C. C. Yang \cite{EY} given a second fundamental theorem for non-Archimedean meromorphic functions and small functions with reduced counting functions which makes a tool other than theorem of Yamanoi in complex analysis to derive a uniqueness theorem for non-Archimedean meromorphic functions sharing $7$ small functions ignoring multiplicities. However, their results are not sharp.

Our aim in this paper is to contruct a better second main theorem for small functions to study the uniqueness problem for meromorphic functions sharing small functions. First, we give a second main theorem for small functions, which improves Theorem 2 in \cite{EY} by increasing the coefficent $\frac{q}{3}$ in front of the characteristic function to $\frac{2q}{5}.$ Our first result is as follows.
\begin{theorem}\label{th1}
 Let $f$ be a nonconstant meromorphic function on $\mathbf{K}.$ Let $a_1,\dots,a_q\,\, (q\geq 5)$ be $q$ distinct small functions with respect to $f.$ We have $$ \frac{2q}{5}T(r,f)\leq\sum_{i=1}^q \overline{N}\big (r,\frac{1}{f-a_i}\big)+S(r,f).$$
\end{theorem}

As an application of Theorem~\ref{th1}, we get a uniquenes theorem for the meromorphic functions sharing a few small functions as follows.
\begin{theorem}\label{th2} Let $f$ and $g$ be two nonconstant meromorphic functions on $\mathbf{K}.$ Let $a_1,\dots,a_q\quad (q\geq 5)$ be $q$ distinct small  functions with respect to $f$ and $g.$ Let $k_1,\dots,k_q$ be $q$ positive integers or $+\infty$ with 
$$ \sum_{j=1}^q \frac{1}{k_j+1}<\frac{2q(q-4)}{5(q+4)}.$$
If $$ \overline{E}(a_j,k_j,f)=  \overline{E}(a_j,k_j,g)\quad (j=1,\dots,q),$$ then $f\equiv g.$
\end{theorem}

In the case $k_1=\dots=k_q=k,$ we can get the result with slightly smaller multiples as follows.
\begin{theorem}\label{th3} Let $f$ and $g$ be two nonconstant meromorphic functions on $\mathbf{K}.$ Let $a_1,\dots,a_q\quad (q\geq 5)$ be $q$ distinct small  functions with respect to $f$ and $g.$ Let $k$ be a positive integer or $+\infty$ with 
$k>\frac{3(q+4)}{2(q-4)}.$
If $$ \overline{E}(a_j,k,f)=  \overline{E}(a_j,k,g)\quad (j=1,\dots,q),$$ then $f\equiv g.$
\end{theorem}

By Theorem~\ref{th3}, we obtain the following corollary, which is a uniqueness theorem for non-Archimedean meromorphic functions sharing $5$ small functions ignoring multiplicities. 
\begin{cor}\label{cor1}Let $f$ and $g$ be two nonconstant meromorphic functions on $\mathbf{K}.$ Let $a_1,\dots,a_5$ be $5$ distinct small  functions with respect to $f$ and $g.$ If $f$ and $g$ share $a_j$ ignoring multiplicities $(j=1,\dots,5,)$ then $f\equiv g.$
\end{cor}

Note that this Corollary~\ref{cor1} improves a result of A. Escassut and C. C. Yang \cite[Theorem 3]{EY}, where  the number of small functions is reduced to $5.$

 \section{Preliminary on Nevanlinna Theory for non-Archimedean meromorphic functions}
We recall the following definitions and results (cf. \cite{HY1}). Let $\mathbf{K}$ be an algebraically closed field of arbitrary characteristic, complete with respect to a non-Archimedean absolute value $|.|.$ If $$h(z)=\sum_{j=0}^{\infty}a_nz^n$$ is an entire function on $\mathbf{K},$ then for each real number $r\geq 0,$ we define 
\begin{align*}|h|_r=\sup_j|a_j|r^j&=\sup\{|h(z)|:z\in\mathbf{K} \text{ with } |z|\leq r\}\\
&=\sup\{|h(z)|:z\in\mathbf{K} \text{ with } |z|= r\}.
\end{align*}

Let $n\big(r,\frac{1}{h}\big)$ denote the number of zeros of $h$ in $\{z| |z|<r\},$ counting multiplicity. Define the {\it valence function} of $h$ by
$$ N\big(r,\frac{1}{h}\big)=\int_0^r\frac{n\big(r,\frac{1}{h}\big)-n\big(0,\frac{1}{h}\big)}{t} dt+n\big(0,\frac{1}{h}\big)\log r,$$ where $n\big(0,\frac{1}{h}\big)$ is order of zero of $h$ at $z=0.$ 

A {\it non-Archimedean meromorphic function} $f$ on $\mathbf{K}$ is the quotient of two non-Archimedean entire functions $\frac{h}{g}$ such that $h,g$ do not have common zeros on $\mathbf{K}.$ Therefore, we can uniquely extend $|.|_r$ to the non-Archimedean meromorphic function $f=\frac{h}{g}$ by defining
$$ |f|_r =\frac{|h|_r}{|g|_r}.$$ Take $a\in K$ and the {\it counting function} $n\big(r,\frac{1}{f-a}\big)$ of $f$ for $a$ by
\begin{equation*}n\big(r,\frac{1}{f-a}\big)=\begin{cases}n(r,f)=n\big(r,\frac{1}{g}\big) & :\,\, a=\infty\\
 n\big(r,\frac{1}{h-ag}\big)& :\,\, a\ne\infty,
\end{cases}
\end{equation*}
and define the {\it valence function} $N\big(r,\frac{1}{f-a}\big)$ of $f$ for $a$ by
\begin{equation*}N\big(r,\frac{1}{f-a}\big)=\begin{cases}N(r,f)=N\big(r,\frac{1}{g}\big) & :\,\, a=\infty\\
 N\big(r,\frac{1}{h-ag}\big)& :\,\, a\ne\infty.
\end{cases}
\end{equation*}
Similarly, we can define $\overline{N}(r,f)$ and $\overline{N}\big(r,\frac{1}{f-a}\big).$ 

Let $f$ be a nonconstant non-Archimedean meromorphic function, $a$ be a small function with respect to $f,$ and $k$ be a positive integer. We denote by $\overline{N}_{k)}\big(r,\frac{1}{f-a}\big)$ the counting function of zeros of $f-a$ with multiplicities at most $k,$ by $\overline{N}_{(k+1)}\big(r,\frac{1}{f-a}\big)$ the counting function of zeros of $f-a$ with multiplicities at least $k+1,$ where each multiple zero in these counting functions counted only once.

We define the {\it compensation function} by $$m(r,f)=\log^+|f|_r =\max\{0,\log |f|_r\}, $$ and the {\it characteristic function} $$ T(r,f)=m(r,f)+N(r,f). $$

The logarithmic derivative lemma can be stated as follows (see \cite{HY1}).

\begin{lemma}[Logarithmic Derivative Lemma]  Let $f$ be a non-constant meromorphic function on $\mathbf{K}.$ Then for any integer $k>0,$ we have  $$ m\big(r,\frac{f^{(k)}}{f}\big)=O(1) $$ as $r\rightarrow\infty.$
 \end{lemma}

We state the first and second fundamental theorem in Nevanlinna theory (see e.g. \cite{HY1}):

\begin{theorem}[The First Main Theorem] Let $f(z)$ be a no-Archimedean meromorphic function and $c\in \mathbf{K}.$ Then $$ T(r,\frac{1}{f-c})=T(r,f)+O(1). $$
\end{theorem}

\begin{theorem}[Second fundamental theorem] \label{sm} Let $a_1,\cdots,a_q$ be a set of distinct numbers of $\mathbf{K}.$ Let $f$ be a non-constant meromorphic function on $\mathbf{K}.$ Then, the inequality $$(q-2)T(r,f)\leq\sum_{j=1}^q\overline{N}\big(r,\frac{1}{f-a_j}\big)-\log r+O(1). $$
\end{theorem}

\section{Proof of Theorem~\ref{th1}}
We first consider the following lemma. 
\begin{lemma}\label{lm1}Let $f$ be a nonconstant meromorphic function on $\mathbf{K}.$ Let $a_1,\dots,a_5$ be distinct small functions with respect to $f.$ We have $$ 2T(r,f)\leq\sum_{i=1}^5 \overline{N}\big (r,\frac{1}{f-a_i}\big)+S(r,f).$$
\end{lemma}
\begin{proof}
By the transformation $$ F=\frac{f-a_2}{f-a_1}\cdot\frac{a_3-a_1}{a_3-a_2},$$ we just need to prove the theorem in the case that $a_1=\infty,a_2=0,a_3=1,$ $a_4,a_5\not\equiv 0,1,\infty,a_4\not\equiv a_5.$ If one of $a_4$ and $a_5$ is constant, then we need to prove nothing according to the second main theorem for constants. Thus, we may assume that both $a_4$ and $a_5	$ are nonconstant small functions of $f.$ Set
\begin{equation}\label{1}
H=\begin{vmatrix}
ff'&f'&f(f-1)\\
a_4a'_4&a'_4&a_4(a_4-1)\\
a_5a'_5&a'_5&a_5(a_5-1)
\end{vmatrix}
\end{equation}
By a simple computation, we get
\begin{align}\label{2}
H=&f(f-1)a_4(a_4-1)a_5(a_5-1)\Big[\big(\frac{a'_4}{a_4}-\frac{a'_5}{a_5}\big)\big(\frac{f'}{f-1}-\frac{a'_5}{a_5-1}\big)\notag\\
&-\big(\frac{a'_4}{a_4-1}-\frac{a'_5}{a_5-1}\big)\big(\frac{f'}{f}-\frac{a'_5}{a_5}\big)\Big].
\end{align}
We claim that $H\not\equiv 0.$ Indeed, on the contrary, assume that $H\equiv 0.$ Since $f$ is not constant and $a_4,a_5\not\equiv 0,1,$ it follows from \eqref{1} that 
\begin{equation}\label{3} 
\Big(\frac{a'_4}{a_4}-\frac{a'_5}{a_5}\Big)\frac{f'}{f-1}-\Big(\frac{a'_4}{a_4-1}-\frac{a'_5}{a_5-1}\Big)\frac{f'}{f}\equiv \Big(\frac{a'_4}{a_4}-\frac{a'_5}{a_5}\Big)\frac{a'_5}{a_5-1}-\Big(\frac{a'_4}{a_4-1}-\frac{a'_5}{a_5-1}\Big)\frac{a'_5}{a_5}.
\end{equation}

We now distinguish four cases

{\it Case 1.} $\frac{a'_4}{a_4}\equiv\frac{a'_5}{a_5}.$ It follows from \eqref{3} that $\frac{a'_4}{a_4-1}\equiv\frac{a'_5}{a_5-1}$ or $\frac{f'}{f}\equiv\frac{a'_5}{a_5}.$ If $\frac{a'_4}{a_4-1}\equiv\frac{a'_5}{a_5-1}$ then $a_4$ and $a_5$ are constants, which contradicts our assumption. This means $\frac{f'}{f}\equiv\frac{a'_5}{a_5}.$ Hence, we get $f=ca_5,$ where $c$ is a constant. This is a contradiction.

{\it Case 2.} $\frac{a'_4}{a_4-1}\equiv\frac{a'_5}{a_5-1}.$ By an argument similar to Case 1, we also get a contradiction.

{\it Case 3.} $\frac{a'_4}{a_4}-\frac{a'_5}{a_5}\equiv \frac{a'_4}{a_4-1}-\frac{a'_5}{a_5-1}\not\equiv 0.$ It follows from \eqref{3} that $$\frac{f'}{f-1}-\frac{f'}{f}\equiv\frac{a'_5}{a_5-1}-\frac{a'_5}{a_5},$$ which implies $$  \frac{f-1}{f}\equiv C \frac{a_5-1}{a_5},$$ where $C$ is a constant. Thus, we obtain $$ \frac{1}{f}\equiv 1- C \frac{a_5-1}{a_5}.$$ It follows that $$ T(r,f)=T\big(r,\frac{1}{f}\big)+O(1)=S(r,f). $$ This is a contradiction.

{\it Case 4.} $\frac{a'_4}{a_4}\not\equiv\frac{a'_5}{a_5},$ $\frac{a'_4}{a_4-1}\not\equiv\frac{a'_5}{a_5-1}$ and $\frac{a'_4}{a_4}-\frac{a'_5}{a_5}\not\equiv \frac{a'_4}{a_4-1}-\frac{a'_5}{a_5-1}.$ Then, it follows from \eqref{3}that the zeros of $f-1$ can only occur at the zeros or $1-$points or the poles of $a_j\quad (j=4,5)$ or the zeros of $\frac{a'_4}{a_4}-\frac{a'_5}{a_5}.$ Similar, the zeros of $f$ can only occur at the zeros or $1-$points or the poles of $a_j\quad (j=4,5)$ or the zeros of $\frac{a'_4}{a_4-1}-\frac{a'_5}{a_5-1}.$ 
Furthermore, from \eqref{3}, we also can be seen that the poles of $f$ can only occur at the zeros or $1-$points or the poles of $a_j\quad (j=4,5)$ or the zeros of $\frac{a'_4}{a_4}-\frac{a'_5}{a_5}- \frac{a'_4}{a_4-1}+\frac{a'_5}{a_5-1}.$ Therefore, we get
\begin{equation}\label{4}
\overline{N}(r,f)+\overline{N}\big(r,\frac{1}{f}\big)+\overline{N}\big(r,\frac{1}{f-1}\big)=S(r,f).
\end{equation}
By \eqref{4} and applying the Second Main Theorem for $f$ and $0,1,\infty,$ we have
\begin{equation*}T(r,f)\leq \overline{N}(r,f)+\overline{N}\big(r,\frac{1}{f}\big)+\overline{N}\big(r,\frac{1}{f-1}\big)-\log r+O(1)=S(r,f).
\end{equation*} This is a contradiction again. Thus, we must have $H\not\equiv 0.$

Given a real number $0 < r < \infty.$ Let $$ \delta(r)=\min\{1,|a_4|_r,|a_5|_r,|a_4-1|_r,|a_5-1|_r,|a_4-a_5|_r\}.$$ Then, we have 
\begin{align*}
\log^+\frac{1}{\delta(r)}&\leq\log^+\max\{1,\frac{1}{|a_4|_r},\frac{1}{|a_5|_r},\frac{1}{|a_4-1|_r},\frac{1}{|a_5-1|_r},\frac{1}{|a_4-a_5|_r}\}\\
&\leq \log^+\Big(1+\frac{1}{|a_4|_r}+\frac{1}{|a_5|_r}+\frac{1}{|a_4-1|_r}+\frac{1}{|a_5-1|_r}+\frac{1}{|a_4-a_5|_r}\Big)\\
&\leq\log^+\frac{1}{|a_4|_r}+\log^+\frac{1}{|a_5|_r}+\log^+\frac{1}{|a_4-1|_r}+\log^+\frac{1}{|a_5-1|_r}\\
&\quad+\log^+\frac{1}{|a_4-a_5|_r}+\log 6\\
&=m\big(r,\frac{1}{a_4}\big)+m\big(r,\frac{1}{a_5}\big)+m\big(r,\frac{1}{a_4-1}\big)+m\big(r,\frac{1}{a_5-1}\big)\\
&\quad+m\big(r,\frac{1}{a_4-a_5}\big)+\log 6.
\end{align*} Hence, we get 
\begin{equation}\label{5}\log^+\frac{1}{\delta(r)}=S(r,f).\end{equation}

We first consider the case that $$ |f-a_j|_r >\frac{1}{2}\delta(r),$$ for all $2\leq j\leq 5.$ In this case,
\begin{align}\label{f1}
m\big(r,\frac{1}{f}\big)+m\big(r,\frac{1}{f-1}\big)+m\big(r,\frac{1}{f-a_4}\big)+m\big(r,\frac{1}{f-a_5}\big)&<5\log^+\frac{1}{\delta(r)}+O(1)\notag\\
&=S(r,f).
\end{align}

Now let $i,\,2\leq i\leq 5,$ be the index among $\{2,3,4,5\}$ such that $$ |f-a_i|_r\leq\frac{1}{2}\delta(r). $$ Then for any $j\ne i,\,2\leq j\leq 5,$ we have $$ \delta(r) \leq |a_i-a_j|_r\leq |f-a_i|_r+|f-a_j|_r\leq \frac{1}{2}\delta(r)+|f-a_j|_r,$$ so $$ |f-a_j|_r\geq \frac{1}{2}\delta(r).$$ Therefore, for $j\ne i,$ we have 
\begin{align*}
\sum_{\substack{j=2\\ j\ne i}}^5 m\big(r,\frac{1}{f-a_j}\big)=\sum_{\substack{j=2\\j\ne i}}^5 \log^+\frac{1}{|f-a_j|_r}\leq 3\log^+\frac{1}{\delta(r)}.
\end{align*} Combining \eqref{5} and the above inequality, we get 
 \begin{equation}\label{6}\sum_{\substack{j=2\\ j\ne i}}^5 m\big(r,\frac{1}{f-a_j}\big)=S(r,f).\end{equation}
On the other hand, for $2\leq i\leq 5,$ we can write
\begin{align*}
&ff'=(f-a_i)(f'-a'_i)+a'_i(f-a_i)+a_i(f'-a'_i)+a_ia'_i,\\
&f'=(f'-a'_i)+a'_i,\\
&f(f-1)=f^2-f=(f-a_i)^2+(2a_i-1)(f-a_i)+a_i^2-a_i.
\end{align*}
By substituting the above equalities into \eqref{1} and using the determinant's properties, we get
\begin{equation}\label{7}
H=\begin{vmatrix}
g_i&f'-a'_i&h_i\\
a_4a'_4&a'_4&a_4(a_4-1)\\
a_5a'_5&a'_5&a_5(a_5-1)
\end{vmatrix}, 
\end{equation} where $$g_i=(f-a_i)(f'-a'_i)+a'_i(f-a_i)+a_i(f'-a'_i),$$ $$h_i=(f-a_i)^2+(2a_i-1)(f-a_i)$$ for $2\leq i\leq 5,$ note that $a_2=0, a_3=1.$
By the definition of $\delta(r),$ we have $\delta(r)\leq 1+|a_i|_r.$ Hence, $$ \log^+\delta(r)\leq \log^+(1+|a_i|_r)\leq \log^+|a_i|_r+\log 2=m(r,a_i)+\log 2=S(r,f).$$
Thus, it follows from \eqref{7} and the Logarithmic Derivative Lemma that 
\begin{align*}
\log^+\Big|\frac{H}{f-a_i}\Big|_r&\leq\log^+\Big |\frac{f'-a'_i}{f-a_i}\Big |_r+\log^+|f-a_i|_r\\
&\quad+O(\log^+|a_i|_r+\log^+|a'_i|_r+\log^+|a_4|_r+\log^+|a'_4|_r\\
&\quad+\log^+|a_5|_r+\log^+|a'_5|_r)\\
&\leq m\big (\frac{f'-a'_i}{f-a_i}\big )+\log^+\delta(r)+S(r,f)\\
&=S(r,f).
\end{align*}
Hence, we get
\begin{align}\label{8}
m\big(r,\frac{1}{f-a_i}\big)&=\log^+\frac{1}{|f-a_i|_r}\leq \log^+\Big|\frac{H}{f-a_i}\Big|_r+\log^+\Big|\frac{1}{H}\Big|_r\notag\\
&\leq m\big(r,\frac{1}{H}\big)+S(r,f).
\end{align}
It follows from \eqref{f1}, \eqref{7} and \eqref{8} that in any case, we have
\begin{align}\label{9}
m\big(r,\frac{1}{f}\big)+m\big(r,\frac{1}{f-1}\big)+m\big(r,\frac{1}{f-a_4}\big)+m\big(r,\frac{1}{f-a_5}\big)\leq m\big(r,\frac{1}{H}\big)+S(r,f).
\end{align}
Hence, by the First Main Theorem, we get
\begin{align}\label{10}
4T(r,f)&\leq N\big(r,\frac{1}{f}\big)+N\big(r,\frac{1}{f-1}\big)+N\big(r,\frac{1}{f-a_4}\big)+N\big(r,\frac{1}{f-a_5}\big) \notag\\
&\quad+T(r,H)-N\big(r,\frac{1}{H}\big)+S(r,f).
\end{align}
On the other hand, suppose that $z_0$ be a zero of $f-a_i,$ $(2\leq i\leq 5)$ of order $s>1$ which is not a pole of $a_4$ or $a_5.$ Then, it follows from \eqref{7} that $z_0$ is also a zero of $H$ of order at least $s-1.$ Hence, from \eqref{10} and the above observations, we get
\begin{align}\label{11}
4T(r,f)&\leq \overline{N}\big(r,\frac{1}{f}\big)+ \overline{N}\big(r,\frac{1}{f-1}\big)+ \overline{N}\big(r,\frac{1}{f-a_4}\big)+ \overline{N}\big(r,\frac{1}{f-a_5}\big) \notag\\
&\quad+T(r,H)+S(r,f).
\end{align}
From \eqref{2}, we have 
\begin{align*}
m(r,H)&\leq 2m(r,f)+S(r,f),\\
N(r,H)&\leq 2N(r,f)+\overline{N}(r,f)+S(r,f).
\end{align*}
Hence, we get
\begin{equation}\label{12}
T(r,H)\leq 2T(r,f)+\overline{N}(r,f)+S(r,f).
\end{equation}
Combining \eqref{11} and \eqref{12}, we obtain
$$2T(r,f)\leq \overline{N}\big(r,\frac{1}{f}\big)+ \overline{N}\big(r,\frac{1}{f-1}\big)+ \overline{N}\big(r,\frac{1}{f-a_4}\big)+ \overline{N}\big(r,\frac{1}{f-a_5}\big)+S(r,f). $$ This completes the proof of Lemma~\ref{lm1}.
\end{proof}

\begin{proof}[Proof of Theorem~\ref{th1}.] 
By Lemma~\ref{lm1}, for every subset $\{i_1,\dots,i_5\}$ of $\{1,\dots,q\}$ such that $1\leq i_1<\dots< i_5\leq q,$ we have 
\begin{equation}\label{t1} 2T(r,f)\leq\sum_{s=1}^5 \overline{N}\big (r,\frac{1}{f-a_{i_s}}\big)+S(r,f).\end{equation}
It is easily seen that the number of such inequalities is $\mathrm{C}_q^5.$ Taking summing up of \eqref{t1} over all subsets $\{i_1,\dots,i_5\}$ of $\{1,\dots,q\}$ as above, we get
\begin{align}\label{t2} 2\mathrm{C}_q^5T(r,f)\leq&\sum_{\substack{\{i_1,\dots,i_5\}\subset\{1,\dots,q\}\\ 1\leq i_1<\dots< i_5\leq q}}\Big( \overline{N}\big (r,\frac{1}{f-a_{i_1}}\big)+\overline{N}\big (r,\frac{1}{f-a_{i_2}}\big)+\overline{N}\big (r,\frac{1}{f-a_{i_3}}\big)\notag\\
&+\overline{N}\big (r,\frac{1}{f-a_{i_4}}\big)+\overline{N}\big (r,\frac{1}{f-a_{i_5}}\big)\Big )+S(r,f).\end{align}
In \eqref{t2}, for each index $i_k,$ the number of terms $\overline{N}\big (r,\frac{1}{f-a_{i_k}}\big)$ is $\mathrm{C}_{q-1}^4.$ Hence, from \eqref{t2}, we get
$$ 2\mathrm{C}_q^5T(r,f)\leq\mathrm{C}_{q-1}^4\sum_{i=1}^q \overline{N}\big (r,\frac{1}{f-a_i}\big)+S(r,f). $$
It follows that $$ \frac{2q}{5}T(r,f)\leq\sum_{i=1}^q \overline{N}\big (r,\frac{1}{f-a_i}\big)+S(r,f).$$ This completes the proof of Theorem~\ref{th1}.
\end{proof}

\section{Proof of Theorem~\ref{th2}}
To prove Theorem~\ref{th2}, we need to prove the following lemma.
\begin{lemma}\label{lm2}
Let $f$ and $g$ be nonconstant meromorphic functions on $\mathbf{K}$ and $a_1,\dots,a_q$ be $q$ distinct small functions with respect to $f$ and $g.$ Let $k_1,\dots,k_q$ be $q$ positive integers or $+\infty.$ Suppose that $$ \overline{E}(a_j,k_j,f)=  \overline{E}(a_j,k_j,g)\quad (j=1,\dots,q).$$ If $f\not\equiv g,$ then for every subset $\{i_1,i_2,i_3,i_4\}$ of $\{1,\dots,q\},$ we have
\begin{align*}\sum_{j\in\{1,\dots,q\}\setminus\{i_1,\dots,i_4\}}\overline{N}_{k_j)}\big (r,\frac{1}{f-a_j}\big )\leq&\sum_{s=1}^4\Big(\overline{N}_{(k_{i_s}+1}\big(r,\frac{1}{f-a_{i_s}}\big)+\overline{N}_{(k_{i_s}+1}\big(r,\frac{1}{g-a_{i_s}}\big)\Big) \\
&+S(r,f)+S(r,g).\end{align*}
\end{lemma}
\begin{proof}
Without losing generality, we just need to prove that 
\begin{align}\label{2.1}\sum_{i=5}^q\overline{N}_{k_i)}\big (r,\frac{1}{f-a_i}\big )\leq&\sum_{j=1}^4\Big(\overline{N}_{(k_j+1}\big(r,\frac{1}{f-a_j}\big)+\overline{N}_{(k_j+1}\big(r,\frac{1}{g-a_j}\big)\Big)\notag\\
& +S(r,f)+S(r,g).\end{align}

If $\sum_{i=5}^q\overline{N}_{k_i)}\big (r,\frac{1}{f-a_i}\big )=S(r,f)+S(r,g),$ then \eqref{2.1} obviously holds. Thus, in the following we may assume that 
\begin{equation}\label{2.2}\sum_{i=5}^q\overline{N}_{k_i)}\big (r,\frac{1}{f-a_i}\big )\ne S(r,f)+S(r,g).\end{equation}
By using the transformation $$ L(w)=\frac{w-a_1}{w-a_2}\cdot\frac{a_3-a_2}{a_3-a_1}$$ and considering two functions $F=L(f), G=L(g)$ if necessary, we may assume that $a_1=0,a_2=\infty, a_3=1$ and $a_4,\dots,a_q$ are distinct small functions with respect to $f$ and $g,$ $a_i\not\equiv 0,1,\infty\,\, (i=4,\dots,q).$

Set 
\begin{equation}\label{2.2}
T:=\frac{f'(a'g-ag')(f-g)}{f(f-1)g(g-a)}-\frac{g'(a'f-af')(f-g)}{g(g-1)f(f-a)},
\end{equation} where $a=a_4.$ Then we have
\begin{equation}\label{2.3}
T=\frac{(f-g)Q}{f(f-1)(f-a)g(g-1)(g-a)},
\end{equation} where 
\begin{align}\label{2.4}
Q&=f'(a'g-ag')(f-a)(g-1)-g'(a'f-af')(g-a)(f-1)\notag\\
&=a'ff'g^2-a'ff'g-a(a-1)ff'g'-aa'f'g^2+aa'f'g-a'f^2gg'\notag\\
&\quad+a'fgg'+a(a-1)f'gg'+aa'f^2g'-aa'fg'.
\end{align}

Suppose that $T\equiv 0.$ Then from \eqref{2.2} we have
\begin{equation}\label{2.5} \frac{f'(a'g-ag')(f-g)}{f(f-1)g(g-a)}\equiv\frac{g'(a'f-af')(f-g)}{g(g-1)f(f-a)}.\end{equation} If $a$ is a constant then $f\equiv g,$ which contradicts our assumption. Thus, $a$ is not a constant. It follows from \eqref{2.5} that 
\begin{equation*}\frac{(f-1)(g-a)}{(g-1)(f-a)}-1\equiv\frac{f'(a'g-ag')}{g'(a'f-af')}-1,\end{equation*} which implies 
\begin{equation*}\frac{(f-g)(1-a)}{(g-1)(f-a)}\equiv\frac{a'[(f'-g')g-(f-g)g']}{g'(a'f-af')}.\end{equation*} This yield that 
\begin{equation}\label{2.6}\frac{f'-g'}{f-g}\equiv\frac{(1-a)g'(a'f-af')}{a'g(g-1)(f-a)}+\frac{g'}{g}.\end{equation} 
It follows from \eqref{2.2} that there exists a point $z_0$ such that $z_0$ is a common zero of $f-a_j$ and $g-a_j$ that is not a zero or a pole of $a,a',a_j,a_j-1,a_j-a$ $(5\leq j\leq q).$ Then, $z_0$ must be a pole of the left hand side of \eqref{2.6}, and not be pole of the right hand side of \eqref{2.6}. This is a contradiction. Thus $T\not\equiv 0.$ 

Suppose that $z_1$ is a common zero of $f-a_j$ and $g-a_j$ that is not a zero or a pole of $a,a_j,a_j-1,a_j-a$ $(5\leq j\leq q).$ Then, it is a zero of $f-g$ and is not a pole of $$ \frac{Q}{f(f-1)(f-a)g(g-1)(g-a)},$$ and hence it is a zero of $T.$ Therefore, since $ \overline{E}(a_j,k_j,f)=  \overline{E}(a_j,k_j,g)\quad (j=1,\dots,q),$ we have 
\begin{align}\label{2.7}
\sum_{i=5}^q\overline{N}_{k_i)}\big (r,\frac{1}{g-a_i}\big )&=\sum_{i=5}^q\overline{N}_{k_i)}\big (r,\frac{1}{f-a_i}\big )\notag\\
&\leq N\big(r,\frac{1}{T}\big)+S(r,f)+S(r,g)\notag\\
&\leq m(r,T)+N(r,T)+S(r,f)+S(r,g).
\end{align}
We estimate $m(r,T).$ From \eqref{2.2} we get
\begin{align}\label{2.8}
T&=\frac{f'}{f-1}\frac{a'g-ag'}{g(g-a)}-\Big(\frac{f'}{f-1}-\frac{f'}{f}\Big)\frac{a'g-ag'}{g-a}\notag\\
&\quad +\frac{g'}{g-1}\frac{a'f-af'}{f(f-a)}-\Big(\frac{g'}{g-1}-\frac{g'}{g}\Big)\frac{a'f-af'}{f-a}\notag\\
&=\frac{f'}{f-1}\Big(\frac{g'}{g}-\frac{g'-a'}{g-a}\Big)-\Big(\frac{f'}{f-1}-\frac{f'}{f}\Big)\Big(a'-a\frac{g'-a'}{g-a}\Big)\notag\\
&\quad \frac{g'}{g-1}\Big(\frac{f'}{f}-\frac{f'-a'}{f-a}\Big)-\Big(\frac{g'}{g-1}-\frac{g'}{g}\Big)\Big(a'-a\frac{f'-a'}{f-a}\Big).
\end{align}
Combining \eqref{2.8} and lemma of the logarithmic derivative, we obtain 
\begin{equation}\label{2.9} m(r,T)=S(r,f)+S(r,g).\end{equation}

Next, we estimate the counting function $N(r,T).$ it follows from \eqref{2.2} that the poles of $T$ can only occur at the zeros of $f,g,f-1,g-1,f-a,g-a$ and the poles of $f,g,a.$ Let $A$ be the set of all zeros, $1-$points and poles of $a.$ We consider all the following possibilities.

If $z\not\in A$ is a common pole of $f$ and $g$ of order $p_1$ and $q_1,$ respectively. It follows from \eqref{2.4} that $z$ is a pole of $Q$ of order at most $2p_1+2q_1+1.$ We have $z$ is a pole of $f-g$ of order $\max\{p_1,q_1\}.$ Hence, from \eqref{2.3} we see that $z$ is a pole of the numerator of $T$ of order at most $ 2p_1+2q_1+1+\max\{p_1,q_1\}$ and it is a pole of the denominator of $T$ of order $3p_1+3q_1.$ Since $$2p_1+2q_1+1+\max\{p_1,q_1\}-(3p_1+3q_1)=1+max\{p_1,q_1\}-(p_1+q_1)\leq 0,$$ it follows that $z$ is not a pole of $T.$

If $z\not\in A$ is a common zero of $f$ and $g$ of order $p_2$ and $q_2,$ respectively. Then $z$ is a zero of $f-g$ of order $\min\{p_2,q_2\}.$ From \eqref{2.4} we have $z$ is a zero of $Q$ of order at least $p_2+q_2-1.$ From \eqref{2.3} we have $$\min\{p_2,q_2\}+p_2+q_2-1-(p_2+q_2)=\min\{p_2,q_2\}-1\geq 0.$$ Hence, $z$ is not a pole of $T.$

If $z\not\in A$ is a common zero of $f-a_i$ and $g-a_i$ with $i=3,4,$ note that $a_3=1, a_4=a.$ Then $z$ is a zero of $f-g.$ On the other hand, we have $z$ is a simple pole of $\frac{f'}{f-a_i}$ and $\frac{g'}{g-a_i}.$ Hence, it follows from \eqref{2.3} that $z$ is not a pole of $T.$ 

If $z\not\in A$ is a pole of $f$ but is not a pole of $g$ and is not a zero of $g,g-1$ and $g-a$ or $z\not\in A$ is a pole of $g$ but is not a pole of $f$ and is not a zero of $f,f-1$ and $f-a.$ It follows from \eqref{2.3} that $z$ is a pole of $T$ of order at most $1.$

If $z\not\in A$ is a zero of $f$ but is not a pole of $g$ and is not a zero of $g,g-1$ and $g-a$ or $z\not\in A$ is a zero of $g$ but is not a pole of $f$ and is not a zero of $f,f-1$ and $f-a.$ Then, it is easily derived from \eqref{2.3} that $z$ is a pole of $T$ of order at most $1.$

If $z\not\in A$ is a zero of $f-a_i$ with $i=3,4$ but is not a pole of $g$ and is not a zero of $g,g-1$ and $g-a$ or $z\not\in A$ is a zero of $g-a_i$ with $i=3,4$ but is not a pole of $f$ and is not a zero of $f,f-1$ and $f-a.$ Then, from \eqref{2.3} we see that $z$ is a pole of $T$ of order at most $1.$

If $z\not\in A$ is a common zero of $f-a_i$ and $g-a_j$ with $1\leq i,j\leq 4,\,\,i\ne j,$ here we regard $f-\infty=\frac{1}{f}$ and $g-\infty=\frac{1}{g}.$ Then, it follows from \eqref{2.3} that $z$ is a pole of $T$ of at most $2.$ 

So, from the assumption and the above observations, we get
\begin{equation}\label{2.10}
N(r,T)\leq \sum_{j=1}^4\Big(\overline{N}_{(k_j+1}\big(r,\frac{1}{f-a_j}\big)+\overline{N}_{(k_j+1}\big(r,\frac{1}{g-a_j}\big)\Big) +S(r,f)+S(r,g).
\end{equation}

By combining \eqref{2.7}, \eqref{2.9} and \eqref{2.10}, we get \eqref{2.1}. The proof of Lemma~\ref{lm2} is completed.
\end{proof}

\begin{proof}[Proof of Theorem~\ref{th2}.] 
Suppose that $f\not\equiv g.$ By Lemma~\ref{lm2}, for every subset $\{i_1,\dots,i_4\}$ of $\{1,\dots,q\},$ we have
\begin{align}\label{th2.1}
&\sum_{j=1}^q\overline{N}_{k_j)}\big (r,\frac{1}{f-a_j}\big )-\sum_{s=1}^4\overline{N}_{k_{i_s})}\big (r,\frac{1}{f-a_{i_s}}\big )\notag\\
&\leq\sum_{s=1}^4\Big(\overline{N}_{(k_{i_s}+1}\big(r,\frac{1}{f-a_{i_s}}\big)+\overline{N}_{(k_{i_s}+1}\big(r,\frac{1}{g-a_{i_s}}\big)\Big)+S(r,f)+S(r,g).
\end{align}
Taking summing up of \eqref{th2.1} over all subsets $\{i_1,\dots,i_4\}$ of $\{1,\dots,q\},$ we get
\begin{align*}
&\mathrm{C}_q^4\sum_{j=1}^q\overline{N}_{k_j)}\big (r,\frac{1}{f-a_j}\big )-\sum_{\substack{\{i_1,\dots,i_4\}\subset\{1,\dots,q\}\\ 1\leq i_1<\dots< i_4\leq q}}\sum_{s=1}^4\overline{N}_{k_{i_s})}\big (r,\frac{1}{f-a_{i_s}}\big )\notag\\
&\leq\sum_{\substack{\{i_1,\dots,i_4\}\subset\{1,\dots,q\}\\ 1\leq i_1<\dots< i_4\leq q}}\sum_{s=1}^4\Big(\overline{N}_{(k_{i_s}+1}\big(r,\frac{1}{f-a_{i_s}}\big)+\overline{N}_{(k_{i_s}+1}\big(r,\frac{1}{g-a_{i_s}}\big)\Big)\notag\\
&\quad+S(r,f)+S(r,g).
\end{align*}
In the above inequality, for each index $i_s,$ the number of terms $\overline{N}\big (r,\frac{1}{f-a_{i_s}}\big)$ is $\mathrm{C}_{q-1}^3.$ Hence, it follows that
\begin{align*}
(q-4)\sum_{j=1}^q\overline{N}_{k_j)}\big (r,\frac{1}{f-a_j}\big )\leq& 4\sum_{j=1}^q\Big(\overline{N}_{(k_{j}+1}\big(r,\frac{1}{f-a_{j}}\big)+\overline{N}_{(k_{j}+1}\big(r,\frac{1}{g-a_{j}}\big)\Big)\notag\\
&+S(r,f)+S(r,g).
\end{align*}
 By an argument similar, we have  
\begin{align*}
(q-4)\sum_{j=1}^q\overline{N}_{k_j)}\big (r,\frac{1}{g-a_j}\big )\leq& 4\sum_{j=1}^q\Big(\overline{N}_{(k_{j}+1}\big(r,\frac{1}{f-a_{j}}\big)+\overline{N}_{(k_{j}+1}\big(r,\frac{1}{g-a_{j}}\big)\Big)\notag\\
&+S(r,f)+S(r,g).
\end{align*}
Hence, we get
\begin{align}\label{th2.2}
&(q-4)\sum_{j=1}^q\Big (\overline{N}\big (r,\frac{1}{f-a_j}\big )+\overline{N}\big (r,\frac{1}{g-a_j}\big )\Big )\notag\\
&\quad\leq (q+4)\sum_{j=1}^q\Big(\overline{N}_{(k_j+1}\big(r,\frac{1}{f-a_j}\big)+\overline{N}_{(k_j+1}\big(r,\frac{1}{g-a_j}\big)\Big)+S(r,f)+S(r,g).
\end{align}
By Theorem~\ref{th1}, we have
\begin{align}\label{th2.3}
\frac{2q}{5}(T(r,f)+T(r,g))&\leq\sum_{j=1}^q\Big (\overline{N}\big (r,\frac{1}{f-a_j}\big )+\overline{N}\big (r,\frac{1}{g-a_j}\big )\Big )\notag\\
&\quad+S(r,f)+S(r,g).
\end{align}
Combining \eqref{th2.2} and \eqref{th2.3}, we get
\begin{align}\label{th2.4}
\frac{2q(q-4)}{5}(T(r,f)+T(r,g))&\leq (q+4)\sum_{j=1}^q\Big(\overline{N}_{(k_j+1}\big(r,\frac{1}{f-a_j}\big)+\overline{N}_{(k_j+1}\big(r,\frac{1}{g-a_j}\big)\Big)\notag\\
&\quad+S(r,f)+S(r,g).
\end{align}

On the other hand, we have
\begin{align}\label{th2.5}
&\sum_{j=1}^q\Big(\overline{N}_{(k_j+1}\big(r,\frac{1}{f-a_j}\big)+\overline{N}_{(k_j+1}\big(r,\frac{1}{g-a_j}\big)\Big)\notag\\
&\leq\sum_{j=1}^q\frac{1}{k_j+1}\Big(N_{(k_j+1}\big(r,\frac{1}{f-a_j}\big)+N_{(k_j+1}\big(r,\frac{1}{g-a_j}\big)\Big)\notag\\
&\leq \sum_{j=1}^q\frac{1}{k_j+1}(T(r,f)+T(r,g)).
\end{align}
The inequalities \eqref{th2.4} and \eqref{th2.5} imply
\begin{equation*}\Big (\frac{2q(q-4)}{5(q+4)}-\sum_{j=1}^q\frac{1}{k_j+1}\Big )(T(r,f)+T(r,g))\leq S(r,f)+S(r,g).
\end{equation*}Hence, when $\sum_{j=1}^q\frac{1}{k_j+1}<\frac{2q(q-4)}{5(q+4)},$ we have a contradiction. Thus, $f\equiv g.$ The proof of Theorem~\ref{th2} is completed.
\end{proof}

\section{Proof of Theorem~\ref{th3}}
First, we prove the following lemma.
\begin{lemma}\label{lm3}
Let $f$ be a nonconstant meromorphic function on $\mathbf{K}.$ Let $a_1,\dots,a_q$ be $q$ distinct small functions with respect to $f.$ Let $k$ be a positive integer or $+\infty.$ Then
$$\sum_{i=1}^q\overline{N}_{(k+1}\big (r,\frac{1}{f-a_i}\big )\leq \frac{3q}{5k}T(r,f)+S(r,f).$$
\end{lemma}
\begin{proof}
For each $1\leq i\leq q,$ we have 
\begin{align*}
k\overline{N}_{(k+1}\big (r,\frac{1}{f-a_i}\big )+\overline{N}\big (r,\frac{1}{f-a_i}\big )&=(k+1)\overline{N}_{(k+1}\big (r,\frac{1}{f-a_i}\big )+\overline{N}_{k)}\big (r,\frac{1}{f-a_i}\big )\\
&\leq N_{(k+1}\big (r,\frac{1}{f-a_i}\big )+N_{k)}\big (r,\frac{1}{f-a_i}\big )\\
&=N\big (r,\frac{1}{f-a_i}\big )\leq T(r,f)+S(r,f).
\end{align*}
Hence, we get 
\begin{equation*}
k\overline{N}_{(k+1}\big (r,\frac{1}{f-a_i}\big )\leq T(r,f)-\overline{N}\big (r,\frac{1}{f-a_i}\big )+S(r,f).
\end{equation*}
Combining this and Theorem~\ref{th1}, we obtain
$$k\sum_{i=1}^q\overline{N}_{(k+1}\big (r,\frac{1}{f-a_i}\big )\leq \frac{3q}{5}T(r,f)+S(r,f).$$ This completes  the proof of Lemma.
\end{proof}

\begin{proof}[Proof of Theorem~\ref{th3}.] 
Suppose that $f\not\equiv g.$ By arguments similar to the inequality \eqref{th2.3} in the proof of Theorem~\ref{th2}, we get
\begin{align}\label{th3.1}
\frac{2q(q-4)}{5}(T(r,f)+T(r,g))\leq &(q+4)\sum_{j=1}^q\Big(\overline{N}_{(k+1}\big(r,\frac{1}{f-a_j}\big)+\overline{N}_{(k+1}\big(r,\frac{1}{g-a_j}\big)\Big)\notag\\
&+S(r,f)+S(r,g).
\end{align}
By Lemma~\ref{lm3}, we have 
\begin{align}\label{th3.2}
\sum_{j=1}^q\Big(\overline{N}_{(k+1}\big(r,\frac{1}{f-a_j}\big)+\overline{N}_{(k+1}\big(r,\frac{1}{g-a_j}\big)\Big)\leq&\frac{3q}{5k}(T(r,f)+T(r,g))\notag\\
&+S(r,f)+S(r,g).
\end{align}
Combining \eqref{th3.1} and \eqref{th3.2}, we obtain
\begin{equation*}\frac{q}{5}\Big (2(q-4)-\frac{3(q+4)}{k}\Big )(T(r,f)+T(r,g))\leq S(r,f)+S(r,g).
\end{equation*} Thus, when $k>\frac{3(q+4)}{2(q-4)},$ we have a contradiction. Hence, $f\equiv g.$  
\end{proof}

 {\it Acknowledgments.}
A part of this article was written while the first author was visiting  Vietnam Institute for Advanced Study in Mathematics (VIASM).  She would like to thank the institute for warm hospitality and partial support.

\end{document}